\documentclass[12pt,reqno]{amsart}
\usepackage{amsthm}
\usepackage[english]{babel}
\usepackage[margin=1.5in]{geometry}
\usepackage[applemac]{inputenc}
\usepackage{tikz}
\usetikzlibrary{arrows,automata}
\usepackage{times}
\usepackage{tikz-cd}
\usepackage[shortlabels]{enumitem}

\newcommand{\C}{{\mathbb C}}

\newcommand{\lb}{\lambda}

\newcommand{{\pf}}{{\bf Proof. }}

\newtheorem{thm}{Theorem}[section]
\newtheorem{corr}[thm]{Corollary}
\newtheorem{lem}[thm]{Lemma}
\newtheorem{prop}[thm]{Proposition}

\newtheorem{deff}[thm]{Definition}
\newtheorem{prob}[thm]{Problem}
\newtheorem{definition}[thm]{Definition}
\usepackage{mathtools}
\newtheoremstyle{case}{}{}{}{}{}{:}{ }{}
\theoremstyle{case}
\newtheorem{claim}{Claim}

\DeclareMathOperator{\tr}{tr}
\DeclareMathOperator{\ran}{ran}

\makeatletter\@addtoreset{equation}{section} \makeatother

\usetikzlibrary{arrows}

\tikzset{
  treenode/.style = {align=center, inner sep=1pt, text centered,
    font=\sffamily},
  arn_n/.style = {treenode, circle, white, font=\sffamily\bfseries, draw=black,
    fill=black, text width=1.5em},
  arn_r/.style = {treenode, circle, black, draw=black,
    text width=1.5em, very thick},
  arn_x/.style = {treenode, rectangle, draw=black,
    minimum width=0.5em, minimum height=0.5em}
}

\begin{document}

\title{The Characteristic Polynomial of Projections}

\author[K. Howell]{Kate Howell$^*$}
\address{Kate Howell: Department of Mathematics and Statistics, SUNY at Albany,
Albany, NY 12222, USA} \email{khhowell@albany.edu}

\author[R. Yang]{Rongwei Yang}
\address{Rongwei Yang: Department of Mathematics and Statistics, SUNY at Albany,
Albany, NY 12222, USA} \email{ryang@math.albany.edu}

\maketitle

\footnotetext{2010 \emph{Mathematics Subject
Classification}: Primary 47A13; Secondary 20E08 and 20Cxx.}
\footnote{\emph{Key words and phrases}: projection matrix, characteristic polynomial, Coxeter group}
\footnote{$^*$ A part of this paper is contained in the doctoral dissertation of the first author.}

\begin{abstract}
This paper proves that the characteristic polynomial is a complete unitary invariant for pairs of projection matrices. Some special cases involving three or more projections are also considered.
\end{abstract}

\section{Introduction}

The characteristic polynomial of a square matrix $A$ is defined by $\det (zI-A)$. It is a basic subject of study in linear algebra and matrix theory, with a wide range of applications in mathematics, science, and engineering. A multivariable version of the characteristic polynomial, on the other hand, has been ostensibly missing until recently. The challenge lies in the difficulty to define a computable notion when involved matrices are noncommuting. Motivated by the early work on group representation theory \cite{Fr,Di75}, determinantal representations \cite{Di21,HV}, and the projective joint spectrum of linear operators \cite{Ya09}, a proper notion of a multivariable characteristic polynomial seems to settle as follows.

 \begin{deff} Given square matrices $A_1, ..., A_n$ of equal size, their characteristic polynomial $Q_A(z):=\det(z_0I+z_1A_1+\cdots+z_nA_n), z=(z_0, ..., z_n)\in \C^{n+1}$. 
 \end{deff}
\noindent Here, the subscript ``$A$'' in $Q_A$ refers to the multiparameter linear pencil (denoted by $A(z)$) inside the determinant. Recent studies have shown that $Q_A$ offers us an efficient tool to describe the algebraic properties of the involved matrices. A few facts will substantiate this claim. First, it is not hard to see that if $Q_A$ is irreducible, then $A_1, ..., A_n$ have no nontrivial common invariant subspace. Further, if $Q_A$ has a linear factor $z_0+\lambda_1z_1+...+\lambda_nz_n$, then $\lambda_j$ is an eigenvalue of $A_j$, $1\le j\le n$. This becomes apparent if one sets variables $z_k=0$ for all $k\neq 0, j$. Some deeper results are as follows.

{\bf 1}. If $A_1, ..., A_n$ are normal matrices, then they commute if and only if $Q_A(z)$ is a product of linear factors \cite{CSZ}.

{\bf 2}. If $A_1, ..., A_n$ are Coxeter generators of a finite Coxeter group, then $Q_A$ determines the group up to isomorphism \cite{CST}.

{\bf 3}. If $G=\{1, g_1, ..., g_n\}$ is a finite group and $A_i=\lb(g_i)$, where $\lb$ is the regular representation of $G$, then there is a one-to-one correspondence between the distinct factors of $Q_A$ and the equivalence classes of irreducible representations of $G$ \cite{Fr}.

{\bf 4}. If ${\mathcal L}$ is an $n$-dimensional complex Lie algebra and $A_1, ..., A_n$ are the adjoint representations of a basis, then $Q_A$ is invariant under the automorphism group of ${\mathcal L}$ \cite{AY}. Moreover, ${\mathcal L}$ is solvable if and only if $Q_A(z)$ is a product of linear factors \cite{HZ19}.

{\bf 5}. Finite dimensional simple Lie algebras can be classified by their characteristic polynomials \cite{GLW}.

For a detailed account of these developments, we refer the reader to \cite{Ya}. This paper aims to study the characteristic polynomial of several projection matrices, and the main result is Theorem \ref{main} which asserts that it is a complete invariant for such pairs. Some relevant issues are also discussed.\\

\noindent {\bf Acknowledgment.} The first author would like to thank her husband, Jason Howell, for his support, encouragement, and prayers. The authors thank the referee for a suggestion which inspired Subsection 2.2.

\section{Preliminaries}

Throughout this paper, we assume the matrix tuples consist of square matrices of the same size. Given a $k\times k$ matrix $A$, we define $c_0( A)=1$ and 
\begin{equation}\label{wedgeA}
c_m(A)=\frac{1}{m!}\det \begin{pmatrix}
                                          \tr A  &   m-1 & 0 & \cdots & 0 \\
\tr A^2  & \tr A&  m-2 & \cdots & 0 \\
 \vdots & \vdots & \vdots & \ddots & \vdots  \\
\tr A^{m-1} & \tr A^{m-2}& \cdots & \cdots & 1    \\ 
\tr A^m  & \tr A^{m-1}& \cdots  & \cdots & \tr A
\end{pmatrix},\ \ 1\leq m\leq k.
\end{equation}
Then the {\em Plemelj-Smithies formula} \cite{GGK}, expresses the characteristic polynomial of $A$ in the form 
\begin{equation}\label{cpexpand}
\det (\lb+A)=\sum_{m=0}^k \lb^{k-m}c_m(A).
\end{equation}
Hence, for matrices $A_1, \dots, A_n\in M_k(\C)$, if we set $A_*(z')=z_1A_1+\cdots +z_nA_n$, where $z'=(z_1, ..., z_n)$, the function $Q_A(z_0, z')$ is the classical characteristic polynomial of the matrix $A_*(z')$. The following fact is thus immediate.
\begin{prop}\label{eigen}
Assume $A_j, B_j\in M_k(\C), 1\leq j\leq n$, and $Q_A=Q_B$. Then for each $z'\in \C^n$ the matrices
$A_*(z')$ and $B_*(z')$ have the same set of eigenvalues counting multiplicity.
\end{prop}
\noindent  One observes that in this case the classical spectra $\sigma(A_j)=\sigma(B_j)$ for each $j$.
 Expanding $Q_A$ as a polynomial in $z_0$, one may write
\begin{equation}\label{Q_A}
Q_A(z)=\sum_{m=0}^{k}z_0^{k-m}q_m(z^{'}),\ \ \ z\in \C^{n+1},
\end{equation}
where $q_m(z')=c_m(A_{*}(z'))$. The following lemma is thus obtained by applying the Plemelj-Smithies formula to $A_*(z')$.
\begin{lem}\label{q1,q2}
For any $k\times k$ matrices $A_1,...,A_n$, we have
\begin{enumerate}[a)]
\item $q_1(z')=\sum^{n}_{j=1}z_j\tr A_j$.
\item $q_2(z')=\frac{1}{2}\sum^{n}_{i,j=1}z_iz_j(\tr A_i\tr A_j-\tr(A_iA_j))$.
\end{enumerate}
\end{lem}
\noindent With greater effort, $q_m$ for $m\geq 3$ can also be determined in a similar way. However, a more elegant description of $q_m$ can be given by the cofactor matrix. For a general matrix $A=(a_{ij})\in M_k(\C)$, we let $C = (C_{ij})$ be its cofactor matrix, i.e., $C_{ij} = (-1)^{i+j}\det(A_{ij})$, where $A_{ij}$ is the $(i,j)$th minor of $A$. The cofactor theorem states that
\begin{equation}\label{cofactor} 
\det A = \displaystyle \sum_{m=1}^{k} a_{im}C_{im},\ \ \ 1\leq i\leq k.
\end{equation} Treating $\det A$ as a multivariable function in its entries $a_{ij}, 1\leq i, j\leq k$, we obtain
\begin{align*}
\dfrac{\partial \det A}{\partial a_{ij}}&= \sum_{m=1}^k\bigg( \dfrac{\partial a_{im}}{\partial a_{ij}} C_{im} + a_{im}\dfrac{\partial C_{im}}{\partial a_{ij}}\bigg)\\
&=C_{ij}+ \sum_{m=1}^ka_{im}\dfrac{\partial C_{im}}{\partial a_{ij}}.
\end{align*}
Since $C_{im}$ is independent of the $i$th row of $A$, we have $\dfrac{\partial C_{im}}{\partial a_{ij}}=0$ for every $m$. It follows that $\dfrac{\partial \det A}{\partial a_{ij}}=C_{ij}, 1\leq i, j\leq k$. Thus one obtains the equality
\begin{equation}\label{ddet}
d \det(A) = \displaystyle \sum_{i,j=1}^{n} C_{ij}da_{ij}=\tr(C^{T}dA)
\end{equation}
\noindent which is well-known. In the case $A$ is invertible, dividing (\ref{ddet}) by $\det A$, we have the Jacobi's formula 
\begin{equation}\label{jacobi}
d\log \det A=\tr \left(A^{-1}dA\right).
\end{equation}

Denoting the cofactor matrix for the linear pencil $A(z)$ by $C_A(z)$, we have the following description of the polynomials $q_1, ..., q_k$ in the expansion (\ref{Q_A}).
\begin{prop}
For $1\leq m \leq k-1$, we have 
\[q_{k-m}(z')=\frac{1}{m!}\Bigg(\frac{\partial^{m-1}}{\partial z_0^{m-1}}\tr C_A(z)\Bigg)\Bigg|_{z_0=0}.\]
\end{prop}
\begin{proof}
In (\ref{ddet}), we substitute the general matrix $A$ by the linear pencil $A(z)$ to obtain \[d Q_A(z) =\tr\left(C_A^{T}(z)(dz_0+A_1dz_1+\cdots +A_ndz_n)\right).\]
It follows that
\begin{equation}\label{dmdet}
\frac{\partial Q_A(z)}{\partial z_0}=\tr C_A^{T}(z)=\tr C_A(z).
\end{equation}
Setting $z_0=0$, we obtain $q_{k-1}(z')=\tr C_A(z) \big|_{z_0=0}$. In view of the expansion (\ref{Q_A}), the formula for $q_{k-m}(z')$ can be obtained by taking the $(m-1)$th partial derivative of (\ref{dmdet}) with respect to $z_0$ and then setting $z_0=0$.
\end{proof}
\noindent Therefore, the polynomial $\tr C_A(z)$ is a generating function for the polynomials $q_j$.


\section{Two Projections}

Projections are used very widely in linear algebra, operator theory, and operator algebras. Besides being an interesting subject in its own right, they often serve as a test ground for researchers to gain intuitions and insights for various other purposes. It is thus necessary to study the characteristic polynomial for projection matrices. This section proves the main theorem of the  paper. 

\subsection{Unitary Equivalence}
We begin with the following definition.
    \begin{definition}
    Two tuples of matrices $(A_1,..., A_n)$ and $(B_1,..., B_n)$ are \textit{unitarily equivalent} if there exists a unitary matrix $U$ such that $UA_i=B_iU$ for all $i$.
    \end{definition}

The problem of finding a complete set of numerical unitary invariants for a single square matrix was solved by Specht \cite{Sp} in terms of trace on the {\em words} in $A$ and $A^*$. For tuples, it was solved similarly by Procesi \cite[Theorem 7.1]{Pro}. 
\begin{thm}\label{CD}
Two $n$-tuples $A$ and $A'$ of matrices are unitarily equivalent if and only if the trace of any finite word in $A_1, ..., A_n, A_1^*, ..., A_n^*$ agrees with that of the corresponding word in $A'_1, ..., A'_n, {A'_1}^*, ..., {A'_n}^*$.
\end{thm}
\noindent An elegant proof of the theorem can be found in Cowen-Douglas \cite{CD}. A complex square matrix $p$ is a \textit{projection} if $p^{*}=p$ and $p^2=p$. For pairs of projection matrices $(p_1, p_2)$, the property of trace reduces the above theorem to the following form.
\begin{lem}\label{2tr}
 Two pairs of projection matrices $(p_1,p_2)$ and $(q_1,q_2)$ are unitarily equivalent if and only if $\tr(p_1)=\tr(q_1)$, $\tr(p_2)=\tr(q_2)$, and $\tr[(p_1p_2)^j]=\tr[(q_1q_2)^j]$ for all $j\geq 1$. 
\end{lem}
\noindent Note that if the projections are of size $k\times k$, then due to the Cayley-Hamilton theorem $(p_1p_2)^k$ can be expressed as a linear combination of lower degree terms. Thus, one only needs to check $\tr[(p_1p_2)^j]=\tr[(q_1q_2)^j]$ for $j\leq k-1$. Pairs of projections have been well-studied in the literature. For an extensive treatise on this subject, we refer the readers to \cite{BS}. The main result of this paper is as follows.
\begin{thm}\label{main}
Two pairs of projection matrices $p=(p_1, p_2)$ and $q=(q_1, q_2)$ are unitarily equivalent if and only if $Q_p=Q_q$.
\end{thm}
\begin{proof}
The sufficiency is trivial. For the necessity, we let $A:=z_1p_1+z_2p_2$ and $B:=z_1q_1+z_2q_2$. The first two conditions of Lemma \ref{2tr} follow readily from Lemma \ref{q1,q2}. In fact, since $\tr(p_1)z_1+\tr(p_2)z_2=\tr(q_1)z_1+\tr(q_2)z_2$ for all complex numbers $z_1$ and $z_2$, we must have $\tr(p_i)=\tr(q_i), i=1, 2$. To obtain the third condition and thus the result of unitary equivalence, we will make three claims. Each of these claims may be easily verified by induction, though the notation is rather tedious for the first two and will therefore be excluded from this paper. 
\begin{claim}
All terms in the fully expanded form of $A^{2n}$ reduce to the form $p_iz_1^sz_2^t$, $p_ip_jp_ip_j\cdots p_ip_jz_1^sz_2^t$, or $p_ip_jp_ip_j\cdots p_iz_1^sz_2^t$, where $i,j\in\{1,2\}$, $i\neq j$, $s,t\in\{0,...,2n\}$, and $s+t=2n$.
\end{claim}
\noindent Taking the trace, we see that all coefficients in the fully expanded form of $\tr(A^{2n})$ reduce to $\tr(p_1)$, $\tr(p_2)$, or $\tr[(p_1p_2)^r]$ for some $r$.
\begin{claim}
    $(p_1p_2)^nz_1^nz_2^n$ is the longest word that appears in the fully expanded form of $A^{2n}$.
\end{claim}
\noindent Taking the trace, we see $\tr[(p_1p_2)^n]$ must appear in $\tr(A^{2n})$ as a coefficient.
\begin{claim}
    $\tr[(p_1p_2)^j]=\tr[(q_1q_2)^j]$, for all $j\geq 1.$
\end{claim}
\noindent To prove Claim 3, for the base case of $n=1$, consider $\tr(A^{2})=\tr(B^{2})$. Then \[\tr(p_1)z_1^2+2\tr(p_1p_2)z_1z_2+\tr(p_2)z_2^2=\tr(q_1)z_1^2+2\tr(q_1q_2)z_1z_2+\tr(q_2)z_2^2.\] It follows that $\tr(p_1p_2)=\tr(q_1q_2)$. Now, assume the hypothesis holds for $k\le n$ and consider the fully expanded form of $\tr(A^{2(n+1)})=\tr(B^{2(n+1)})$. From Claim 1, we see that all coefficients must be of the form $\tr(p_1)$, $\tr(p_2)$, or $\tr(p_1p_2)^r$, and from Claim 2 we see that $r\le n+1$. Moreover, Claim 2 indicates that $\tr(p_1p_2)^{n+1}$ must appear as a coefficient at least once. Hence, by factoring like $z^sz^t$ terms and noting that $\tr(p_1)=\tr(q_1)$, $\tr(p_2)=\tr(q_2)$, and $\tr[(p_1p_2)^r]=\tr[(q_1q_2)^r]$, $\forall r\leq n$ by the induction assumption, we must have that $\tr[(p_1p_2)^{n+1}]=\tr[(q_1q_2)^{n+1}]$. This proves Claim 3, and the theorem follows from Lemma \ref{2tr}.
\end{proof}

With Theorem \ref{main} at hand, we realize yet another property of the characteristic polynomial: it is a complete invariant for pairs of projections. The symmetry of $Q_p$ thus has the following consequence.
\begin{corr}
Let $p=(p_1, p_2)$ be a pair of projection matrices. If $Q_p$ is symmetric with respect to $z_1$ and $z_2$, then $p_1$ and $p_2$ have the same rank.
\end{corr}
\begin{proof}
Let $q=(p_2, p_1)$. Then the assumption implies that $Q_p=Q_q$, and hence the corollary follows from Theorem \ref{main}.
\end{proof} 

Now, let us consider the following application. Let $r=2p-I$, where $p$ is a projection. A simple computation will show us that $r$ is an involution, i.e., $r^2=I$. On the other hand, given an involution $r$, the matrix $p=\frac{r+I}{2}$ is a projection. This fact gives the following corollary.
\begin{corr}
Two tuples of involutions $(r_1,r_2)$ and $(s_1,s_2)$ are unitarily equivalent if their characteristic polynomials coincide.
\end{corr}

\subsection{The Structure of Two Projections}

In this subsection, we inherit from Halmos \cite{Hal} the notations $P$ and $Q$ for two projections on a Hilbert space $\mathcal{H}$. Define $L=\ran P$ and $N=\ran Q$. Then we may decompose $\mathcal{H}$ in the following way: \[L=(L\cap N)\oplus (L\cap N^\bot)\oplus M_0 \ \ \textup{and} \ \ L^\bot =(L^\bot\cap N)\oplus(L^\bot \cap N^\bot)\oplus M_1,\] where $M_0$ and $M_1$ are closed subspaces of $L$ and $L^\bot$, respectively. Therefore,
\begin{equation}\label{decom}
\mathcal{H}=(L\cap N)\oplus (L\cap N^\bot)\oplus(L^\bot\cap N)\oplus(L^\bot \cap N^\bot)\oplus M_0\oplus M_1.
\end{equation} Note that if $M_0=M_1=\{0\}$, we have the decomposition \[\mathcal{H}=(L\cap N)\oplus (L\cap N^\bot)\oplus(L^\bot\cap N)\oplus(L^\bot \cap N^\bot),\] and $P=(1,1,0,0)$, $Q=(1,0,1,0)$, where $(\alpha_{00},\alpha_{01},\alpha_{10},\alpha_{11})$ stands for \[\alpha_{00}I_{(L\cap N)}\oplus\alpha_{01}I_{(L\cap N^\bot)}\oplus\alpha_{10}I_{(L^\bot\cap N)}\oplus\alpha_{11}I_{(L^\bot \cap N^\bot)}.\] 
With respect to the decomposition (\ref{decom}), we have the following theorem due to Halmos \cite{Hal,BS}.
\begin{thm}\label{Halmos}
    It holds that $\dim M_0=\dim M_1$, and if the dimension is nonzero, then up to unitary equivalence, 
    \[P=(1,1,0,0)\oplus\begin{pmatrix}
I & 0 \\
0 & 0 
\end{pmatrix},  \ \  Q=(1,0,1,0)\oplus\begin{pmatrix}
H & \sqrt{H(1-H)} \\
\sqrt{H(1-H)} & H 
\end{pmatrix},\]
where the operator $H:=PQP\mid_{M_0}$ is a positive, self-adjoint contraction such that $0\leq H\leq I$ and $\ker H=\ker (I-H)=\{0\}$.
\end{thm}
This theorem enables one to compute the characteristic polynomial $\det (z_0+z_1P+z_2Q)$. Indeed,
 Theorem \ref{Halmos} gives the decomposition 
\[P=\begin{pmatrix}
P_1 & 0 \\
0 & P_2 
\end{pmatrix}, \ \ \textup{and} \ \ Q=\begin{pmatrix}
Q_1 & 0 \\
0 & Q_2 
\end{pmatrix},\] where $P_1=(1,1,0,0)$, $Q_1=(1,0,1,0)$, \[ \ \ P_2=\begin{pmatrix}
I & 0 \\
0 & 0 
\end{pmatrix}, \ \ \textup{and} \ \ Q_2=\begin{pmatrix}
H & \sqrt{H(1-H)} \\
\sqrt{H(1-H)} & I-H 
\end{pmatrix}.\] Here $H=P_2Q_2P_2|_{M_0}$. Then we have 
\begin{align*}
   & \text{det}(z_0I+z_1P+z_2Q)\\
   & = \text{det}\begin{pmatrix}
z_0I+z_1P_1+z_2Q_1 & 0 \\
0 & z_0I+z_1P_2+z_2Q_2 
\end{pmatrix} 
\\ & = \det(z_0I+z_1P_1+z_2Q_1)\det(z_0I+z_1P_2+z_2Q_2).
\end{align*} 
A simple calculation gives
\begin{align}\label{P1Q1}
\det(z_0I+z_1P_1+z_2Q_1)=z_0^{k_1}(z_0+z_1)^{k_2}(z+z_2)^{k_3}(z_0+z_1+z_2)^{k_4},
\end{align}
where $k_1, k_2, k_3$, and $k_4$ are the dimensions of $L^\bot \cap N^\bot, L \cap N^\bot, L^\bot \cap N$, and $L\cap N$, respectively. Furthermore, we have
\[ \text{det}(z_0I+z_1P_2+z_2Q_2)=\text{det} \begin{pmatrix}
z_0I+z_1I+z_2H & z_2\sqrt{H(I-H)} \\
z_2\sqrt{H(I-H)} & z_0I+z_2(I-H)
\end{pmatrix}.\]
\noindent Since all the block entries are commuting and $H$ is diagonalizable, the above determinant is equal to
\begin{align}\label{generic}
&\text{det}\left((z_0I+z_1I+z_2H)(z_0I+z_2(I-H)-z_2^2H(I-H)\right)\nonumber
\\ & =\text{det}\left((z_0^2+z_0z_1+z_0z_2)I+z_1z_2(I-H)\right)\nonumber
\\ &= \prod_{x\in \hat{\sigma}(I-H)}\left(z_0^2+z_0(z_1+z_2)+z_1z_2x\right),
\end{align}
where $\hat{\sigma}(I-H)$ stands for the set of eigenvalues of $I-H$ counting multiplicity. Summarizing the computations above, we can state the following theorem. 
\begin{thm}\label{CPP}
Let $P$ and $Q$ be projection matrices of the same size. Then
\begin{align*}
&\det (z_0+z_1P+z_2Q)\\
&=z_0^{k_1}(z_0+z_1)^{k_2}(z+z_2)^{k_3}(z_0+z_1+z_2)^{k_4}\prod_{x\in \hat{\sigma}(I-H)}\left(z_0^2+z_0(z_1+z_2)+z_1z_2x\right),
\end{align*}
where $(k_1, k_2, k_3, k_4)$ and $H$ are as defined before.
\end{thm}
It is worth noting that an infinite dimensional version of (\ref{generic}) has been given in \cite{GY,WJRQ}. In response to a suggestion of the referee, we now offer an alternative proof of Theorem \ref{main}. If $P'$ and $Q'$ are matrices of the same size and \[\det(z_0I+z_1P+z_2Q)=\det(z_0I+z_1P'+z_2Q'),\] then the decomposition (\ref{decom}) for the pairs $(P, Q)$ and $(P', Q')$ have equal dimensional components. Moreover, the positive matrices $H$ and $H'$ have the same set of eigenvalues counting multiplicity and hence are unitarily equivalent. An application of Theorem \ref{Halmos} then shows that $(P, Q)$ and $(P', Q')$ are unitarily equivalent.

According to Halmos, the first four summands in (\ref{decom}) are ``thoroughly uninteresting". A pair of projections $P$ and $Q$ are said to be in {\em generic position} if the four summands are all trivial. In this case $\ker H=\ker (I-H)=\{0\}$, indicating that $\hat{\sigma}(I-H)\subset (0, 1)$. It is a pleasure to check that the factor $z_0^2+z_0(z_1+z_2)+z_1z_2x$ in Theorem \ref{CPP} is irreducible for every $x\in (0, 1)$. 

\begin{corr}
Let $\Phi(z)$ denote the characteristic polynomial of two projection matrices. Then

a) $\Phi(z)$ is a product of linear and/or irreducible quadratic polynomials;

b) $\Phi(z)$ contains no linear factor if and only if the projections are in generic position.
\end{corr}
\noindent Since the degree of the characteristic polynomial is equal to the size of the involved matrices, part b) exposes the following interesting phenomenon which is essentially due to the fact $\dim M_0=\dim M_1$ in Theorem \ref{Halmos}.
\begin{corr}
No two projections in $M_k(\C)$ are in generic position if $k$ is odd.
\end{corr}

\section{The Coxeter Groups}
We now move to consider the characteristic polynomial of $n$ projections for $n\geq 3$. Interestingly, the Coxeter groups play a useful role here. A {\em Coxeter group} $W$ is generated by a set $S=\{g_1,...,g_n\}$ that satisfies the relations:

 1) $(g_ig_j)^{m_{ij}}=1,\ 1\leq i, j\leq n$, 
 
 2) $m_{ii}=1$, and $2\leq m_{ij}$ when $i\neq j$. 
 
 \noindent The associated $n\times n$ matrix $M=(m_{ij})$ is called the Coxeter matrix, and the pair $(W,S)$ is called the Coxeter system. Observe that $M$ is symmetric since $(g_ig_j)^{-1}=g_jg_i$. Note that the situation $m_{ij}=\infty$ for some $i$ and $j$ is allowed in a Coxeter matrix, in which case $g_i$ and $g_j$ are considered to be free from each other, and the subgroup generated by them is isomorphic to the infinite dihedral group $D_\infty$. Moreover, since $(g_i)^2=1$, the generators $g_i$ can be realized as reflections in a vector space. The Coxeter groups (or Weyl groups in particular) play an important role in group theory, and they are an indispensable tool in the classification of finite dimensional complex simple Lie algebras. We refer the reader to \cite{Hum2} for a treatise on the subject and to \cite{CST} for a recent study of their characteristic polynomials.

\subsection{The Tits Representation}
 Consider a Coxeter system $(W, S)$ and a basis $\{e_1,...,e_n\}$ for $\mathbb{C}^n$. The {\em Tits representation} of $W$, also known as the reflection representation or the geometric representation, is defined as \[\rho(g_i)(e_j)=e_j+2\alpha_{ij}e_i,\ \ \ 1\leq i, j\leq n,\]
 where $\alpha_{ij}=\cos{\frac{\pi}{m_{ij}}}$. Observe that $\alpha_{ii}=-1$, and $\alpha_{ij}\geq 0$ for $i\neq j$. Note in particular that if $m_{ij}$ is infinite then $\alpha_{ij}=1$. With respect to the basis $\{e_1,...,e_n\}$, we have the matrix representation
 \[\rho(g_i)=\begin{pmatrix}
1 & 0 & \hdots & 0\\
0 & 1 & \hdots & 0\\
\vdots & \vdots & \ddots & \vdots\\
\alpha_{i1} & \alpha_{i2} & \hdots & \alpha_{in}\\
\vdots & \vdots & \ddots & \vdots\\
0 & 0 & \hdots & 1
\end{pmatrix},\ \ \ i=1, ..., n.\]
\noindent Evidently, this matrix in general is not unitary with respect to the ordinary Euclidean metric on $\C^n$. However, it is unitary with respect to the inner product induced by the sesquilinear form $B$ on $\mathbb{C}^n$ such that $B(e_i, e_j)=-\alpha_{ij}$. Indeed, for all $1\leq i, j, k\leq n$, one verifies that
\begin{equation*}
\begin{split}
& B(\rho(g_k)e_i,\rho(g_k)e_j)\\
 =& B\left(e_i+2\alpha_{ki}e_k, e_j+2\alpha_{kj}e_k\right)\\
 =& B(e_i,e_j)+2\alpha_{kj}B(e_i,e_k)+2\alpha_{ki}B(e_k,e_j)+4\alpha_{ki}\alpha_{kj}B(e_k,e_k)
 \\ =& B(e_i,e_j)-2\alpha_{kj}\alpha_{ki}-2\alpha_{ki}\alpha_{kj}+4\alpha_{ki}\alpha_{kj}= B(e_i,e_j).
\end{split}
\end{equation*}

Recall from (\ref{Q_A}) that for the matrices $A_i=\rho(g_i), i=1, ..., n$, the characteristic polynomial can be expressed as $Q_A(z)=\sum_{m=0}^{n} z_0^{n-m}q_m(z')$. Using Lemma \ref{q1,q2}, we derive by direct computation the following formulae:
\begin{align}
q_1(z')&=(n-2)(z_1+\cdots + z_n),\nonumber \\
q_2(z')&=\frac{1}{2}(n-1)(n-4)(z_1^2+\cdots+z_n^2)+\sum_{i<j}^{n} z_iz_j(n^2-5n+8-\alpha_{ij}^2).\label{q2}
 \end{align}
This leads to the following theorem.
\begin{thm}\label{cox}
 A Coxeter system $(W,S)$ is uniquely determined, up to isomorphism, by its characteristic polynomial with respect to the Tits representation.
\end{thm}
\begin{proof}
Let $(W,\{g_1,...,g_n\})$ and $(W',\{g'_1,...,g'_k\})$ be two Coxeter systems with the Tits representations $\rho$ and $\rho '$, respectively. Set \[Q_\rho =\det(z_0I+z_1\rho(g_1)+\cdots+\rho_n(g_n))\] and likewise for $Q_{\rho '}$, and assume $Q_\rho=Q_{\rho'}$. It follows that \[k=\deg Q_{\rho '}=\deg Q_\rho=n,\ \ \text{and}\ \ q_2=q'_2.\] Since for each pair $i<j$ we have $(n^2-5n+8-\alpha_{ij}^2)$ as the coefficient of $z_iz_j$ in $q_2$, we must have $\alpha_{ij}^2={\alpha'}_{ij}^2$ and therefore $\alpha_{ij}=\alpha'_{ij}$. But since $\alpha_{ij}=\cos{\frac{\pi}{m_{ij}}}$, we must have that $m_{ij}=m'_{ij}$ for all $i,j\in\{1,...,n\}$. Hence, by sending $g_i$ to $g'_i$, we obtain an isomorphism between $W$ and $W'$.
\end{proof}

\subsection{Projections onto Hyperplanes}
Let us now make a connection between the Tits representation $\rho$ and projections, which will be essential in proving our next result. First, recall that the Tits representation is unitary with respect to the inner product defined by $B(e_i,e_j)=-\cos (\frac{\pi}{m_{ij}})=-\alpha_{ij}$ on $\mathbb{C}^n$. Moreover, for each $i$ since $(\rho(g_i))^2=\rho(g_i^2)=I$, we have that $\rho(g_i)=\rho^{-1}(g_i)=\rho^{*}(g_i)$, i.e. $\rho(g_i)$ is an involution. We set $p_i=\frac{I+\rho(g_i)}{2}$. It is not hard to check that $p_i$ is the orthogonal projection onto the hyperplane $\{v\in \C^n\mid B(e_i, v)=0\}$. 
\begin{thm}\label{hyper}
Suppose $p=(p_1,...,p_n)$ and $p'=(p'_1,...,p'_n)$ are tuples of linearly independent projections onto hyperplanes in $\mathbb{C}^{n}$. Then $p$ and $p'$ are unitarily equivalent if and only if $Q_p=Q_{p'}$.
\end{thm}
\begin{proof}
 The necessity is trivial. To prove the sufficiency, we define $g_i=2p_i-I$ and $g'_i=2p'_i-I$, $1\leq i\leq n$, which are reflection matrices about the hyperplanes $p_i(\C^n)$ and $p'_i(\C^n)$, respectively. Let $W=\langle g_1,..., g_n\rangle$ and $W'=\langle g'_1,..., g'_n\rangle$ be the Coxeter groups with corresponding Coxeter matrices $M=(m_{ij})$ and resp. $M'=(m'_{ij})$. Moreover, let $e_i$ (resp. $e_i')$ be a unit vector in $\ker p_i$ (resp. $\ker q_i$). Then $\{e_1, ..., e_n\}$ (resp. $\{e'_1, ..., e'_n\}$) is linearly independent and hence a basis for $\C^n$. Observe that the identity map is the Tits representation for both $W$ and $W'$ in this case. Suppose $Q_p=Q_p'$. Then $M=M'$ by the proof of Theorem \ref{cox}. Furthermore, we have an isomorphism $\phi:W\rightarrow W'$ defined by $\phi(g_i)=g'_i$. We shall construct a unitary $U$ and show that $U^{*}g'_iU=g_i$, $\forall i\in\{1,...,n\}$, which would therefore imply the unitary equivalence of $p$ and $p'$. To proceed, we define $U(e_i)=e_i', 1\leq i\leq n$. Note that $U$ is unitary since \[B(e_i,e_j)=-\alpha_{ij}=-\alpha'_{ij}=B(e'_i,e'_j).\] Moreover, we have
 \begin{align*}
U^{*}g'_iU(e_j)  & = U^{*}(U(e_j)+2\alpha'_{ij}U(e_i))
 \\ & = e_j+2\alpha_{ij}e_i=g_i(e_j).
\end{align*}
\end{proof}
The following is an immediate consequence.
\begin{corr}
Suppose $p=(p_1,...,p_n)$ and $p'=(p'_1,...,p'_n)$ are tuples of linearly independent rank-$1$ projections in $\mathbb{C}^{n}$. Then they are unitarily equivalent if and only if their characteristic polynomials coincide.  
\end{corr}
\noindent In this case, since $I-p_i$ and $I-p'_i$ are projections onto hyperplanes, the corollary follows directly from Theorem \ref{hyper} after a change of variables in the characteristic polynomial.

\section{Concluding Remarks}

Results in this paper serve as yet another clear evidence that joint characteristic polynomial can play a useful role in the study of multi-matrix systems. This role is in fact outstanding when the matrices are noncommuting because few other effective tools are currently available. Application of this idea to other matrices with algebraic properties is a tempting project, and there is indeed much to anticipate.
In particular, the work here inspires the curiosity about the characteristic polynomial for three or more projection matrices. We end this paper with some natural questions in this regard.

\begin{prob} Consider three projection matrices $p=(p_1, p_2, p_3)$ of equal size.

a) Can their joint characteristic polynomial $Q_p$ be computed?

b) What are the possible degrees of $Q_p$'s irreducible factors?

c) Is $Q_p$ a complete unitary invariant of the tuple $p$?
\end{prob}

\vspace{5mm}

\end{document}